\documentclass[oneside, a4paper, 11pt]{amsart}

\usepackage{amsmath,amsthm,amssymb,bbm,comment,mathrsfs, tikz}

\usepackage{thmtools}

\newtheorem{theorem}[subsubsection]{Theorem}
\newtheorem{lemma}[theorem]{Lemma}

\newtheorem{corollary}[subsubsection]{Corollary}

\theoremstyle{definition}
\newtheorem{definition}[subsubsection]{Definition}

\newtheorem{remark}[theorem]{Remark}

\newtheorem{example}[subsubsection]{Example}

\numberwithin{equation}{section}

\newcommand{\mZ}{\mathbb{Z}}

\newcommand{\bk}{\Bbbk}

\newcommand{\Ind}{\operatorname{Ind}}
\newcommand{\PSh}{\operatorname{PSh}}

\newcommand{\End}{\operatorname{End}}
\newcommand{\Hom}{\operatorname{Hom}}

\newcommand{\Sym}{\mathcal{S}ym}
\newcommand{\Rep}{\operatorname{Rep}}

\newcommand{\bx}{\mathbf{x}}
\newcommand{\bv}{\mathbf{v}}
\newcommand{\mS}{\mathbb{S}}

\newcommand{\cA}{\mathcal{A}}
\newcommand{\cC}{\mathcal{C}}

\newcommand{\cB}{\mathcal{B}}
\newcommand{\cO}{\mathcal{O}}

\newcommand{\op}{\mathrm{op}}
\newcommand{\Vecc}{\mathtt{Vec}}
\newcommand{\sVec}{\mathtt{sVec}}

\newcommand{\unit}{\mathbf{1}}
\newcommand{\mG}{\mathbb{G}}

\newcommand{\cI}{\mathcal{I}}

\newcommand{\ev}{\mathrm{ev}}
\newcommand{\co}{\mathrm{co}}
\newcommand{\id}{\mathrm{id}}

\begin{document}

\title[]{Invertible exterior powers}

\author{Kevin Coulembier}
\address{School of Mathematics and Statistics, University of Sydney, NSW 2006, Australia}
\email{kevin.coulembier@sydney.edu.au}

\maketitle

\begin{abstract}
We present a proof of the fact that in a symmetric monoidal category over a field of characteristic zero, objects with an invertible exterior power are rigid. As an application we prove two recent conjectures on dimensions in symmetric monoidal categories by Baez, Moeller and Trimble and further conjectures by Baez and Trimble.
\end{abstract}

%\tableofcontents

\section{Preliminaries}

We refer to \cite{EGNO} for details on monoidal categories. Let $K$ be a commutative ring. Following \cite{BMT} we will refer to a $K$-linear symmetric monoidal category $(\cC,\otimes,\unit)$ that is idempotent complete and additive as a {\em 2-rig over~$K$}. The braiding on a symmetric monoidal category $\cC$ will be denoted by $\sigma_{X,Y}: X\otimes Y\xrightarrow{\sim}Y\otimes X$ for $X,Y\in\cC$. For $n\in\mZ_{>0}$ and $X\in\cC$ with $\cC$ a 2-rig over a ring $K$ in whch $n!$ is invertible, we denote by $S^nX$ and $\wedge^nX$ the $n$-th symmetric and exterior power, see \cite[\S 8]{BMT}.

Let $\cC$ be a symmetric monoidal category.
A dual of an object $X\in \cC$  is a triple $(X^\ast,\ev_X,\co_X)$ of an object $X^\ast\in \cC$ with morphisms $\ev_X:X^\ast\otimes X\to\unit$ and $\co_X:\unit\to X\otimes X^\ast$ satisfying the `snake relations' $(X\otimes \ev_X)\circ (\co_X\otimes X)=\id_X$ and $(\ev_X\otimes X^\ast)\circ (X^\ast\otimes \co_X)=\id_{X^\ast}$. An object is called rigid if it has a dual (automatically unique up to isomorphism). The categorical dimension $\dim X\in\End(\unit)$ of a rigid $X\in\cC$ is given by $\ev_X\circ \sigma_{X,X^\ast}\circ\co_X$.

 An object $X\in \cC$ is invertible if there exists a $Y\in\cC$ such that $X\otimes Y\simeq \unit$. It then follows that $Y$ can be made into a dual of $X$. If $\cC$ is $\mZ[1/2]$-linear, we then have $\dim X=\pm 1$ (which stands for $\pm 1\cdot\id_{\unit}$), with $\dim X=1$ if and only if $\wedge^2 X=0$ and $\dim X=-1$ if and only if $S^2 X=0$. Following \cite{BMT} we say that in the former case $X$ is a bosonic line object and in the latter case a fermionic line object.

Given any $K$-linear symmetric monoidal category $\cC_0$, one can `Cauchy' complete it canonically into a 2-rig $\cC$ over $K$, by formally adjoining direct summands (idempotents) and direct sums, see \cite[\S 1.2]{AK}. In the idempotent completion of $\cC_0$, objects are given by pairs $(X,e)$, with $X\in\cC$ and $e\in \End(X)$ an idempotent. The morphisms $(X,e)\to (X',e')$ are the morphisms in $\cC$ of the form $e'\circ f\circ e$ for some $f\in \Hom(X,X')$. 

We will use the above notation in general for the summand $\wedge^n X$ of $X^{\otimes n}$ and we denote by $e_n$ the corresponding idempotent, the skew symmetriser $a_n:=\sum_{x\in S_n}(-1)^{|x|}x$ for $S_n$ divided by $n!$. By abuse of notation we thus write $a_n$ and $e_n$ for elements of $\mZ S_n$ and $\mZ[1/n!]S_n$, as well as for their images in $\End(X^{\otimes n})$.

Let $\bk$ be a field for the rest of the preliminaries. We denote by $\Sym_0$ the strict symmetric monoidal category with objects labelled by the natural numbers and endomorphism algebras given by the group algebras $\bk S_n$. Its Cauchy completion $\Sym$ is the universal 2-rig on one object, which we denote by $\bx$, see \cite[\S 1]{BMT}. Concretely, given a 2-rig over $\bk$, the assignment $F\mapsto F(X)$ yields an equivalence between the category of $\bk$-linear symmetric monoidal functors $\Sym\to \cC$ with monoidal natural transformations and the category $\cC$.

For the algebraic group $GL_n$ over $\bk$, we denote by $\Rep_{\bk} GL_n$ its monoidal category of rational representations. If $\bk$ is of characteristic zero, and $V$ denotes the natural $n$-dimensional representation, then $\Rep_{\bk} GL_n$ is the Cauchy completion of its full symmetric monoidal subcategory $[\Rep_{\bk} GL_n]_0$ with objects given by tensor products of $V$ and $V^\ast$. Let $M_n$ be the algebraic monoid of $n\times n$-matrices. Then its representation category $\Rep_{\bk} M_n$ is the full subcategory of $\Rep_{\bk} GL_n$ of polynomial representations. If $\bk$ is of characteristic zero then the 2-rig $\Rep_{\bk} M_n$ is the Cauchy completion of the full symmetric monoidal subcategory $[\Rep_{\bk} M_n]_0$ of $\Rep_{\bk}GL_n$ with objects the tensor powers of $V$.

For a self-contained proof of the following lemma in characteristic zero, see \cite{BT}.
\begin{lemma}\label{LemMn}Let $\bk$ be a field.
$\Rep_{\bk} M_n$ is the universal 2-rig over $\bk$ on one object on which the skew symmetriser $a_{n+1}$ of $\bk S_{n+1}$ vanishes.
\end{lemma}
\begin{proof}
By construction, $\Sym/I_n$ is the universal 2-rig on an object on which $a_{n+1}$ vanishes, for $I_n$ the tensor ideal generated by $a_{n+1}$. It is known that $\Sym/I_n$ is equivalent to $\Rep_{\bk} M_n$. Indeed, that $\Sym_0\to [\Rep M_n]_0$ is full follows from \cite[Theorem~4.1]{DP} and that the kernel is generated by $a_{n+1}$ follows from \cite[Theorem~4.2]{DP}.
\end{proof}

For $t\in \bk$, the category $(\Rep_{\bk} GL)_t$ from \cite[\S 10]{Del07} is the universal\footnote{Contrary to the notion of universality used for $\Sym$ and $\Rep M_n$, this means that there is an equivalence $F\mapsto F(\bx_t)$ between the category of $\bk$-linear symmetric monoidal functors $F:(\Rep GL)_t\to \cC$ and the subcategory of $\cC$ comprising rigid objects of categorical dimension $t$ and {\em isomorphisms} between them. All universal 2-rigs below (except Example~\ref{RemMn}) are to be interpreted in this sense.} 2-rig over $\bk$ on a rigid object $\bx_t$ of categorical dimension $t$ (meaning $t\cdot\id_{\unit}$), see \cite[Proposition~10.3]{Del07}. The category $(\Rep_{\bk} GL)_t$ is the Cauchy completion of the oriented Brauer category, a diagrammatic category which is a variant of $\Sym_0$.

We say that an object $X\in\cC$ is `symmetrically self-dual' if it is its own dual and $\ev_X\circ\sigma_{X,X}=\ev_X$. Similarly, $X$ is skew-symmetrically self-dual if $\ev_X\circ\sigma_{X,X}=-\ev_X$.
The category $(\Rep_{\bk} O)_t$ from \cite[\S 9]{Del07} is the universal 2-rig over $\bk$ on a symmetrically self-dual object $\bv_t$ of categorical dimension $t$, see \cite[Proposition~10.3]{Del07}. The category $(\Rep_{\bk} O)_t$ is the Cauchy completion the Brauer category $\cB(t)$. 

Assume $\mathrm{char}(\bk)\not=2$. The $2$-rig $\sVec$ of finite-dimensional supervector spaces is the category of $\mZ/2$-graded vector spaces where, for $\bar{\unit}$ the one-dimensional space in odd degree, $\sigma_{\bar{\unit},\bar{\unit}}=-1$. In other words, $\bar{\unit}$ is a fermionic line object. For an affine group scheme $G$ (we will not need the standard generality of an affine group superscheme) and a central group homomorphism $\varepsilon :\mZ/2\to G$ we consider the 2-rig $\Rep (G,\varepsilon)$ of representations of $G$ in the category of supervector spaces such that the generator of $\mZ/2$, via restriction along $\varepsilon$, acts by $1$ on even vectors and by $-1$ on 
odd vectors, see \cite[\S 0.3]{Del02}. If $\varepsilon$ is trivial then $\Rep (G,\varepsilon)=\Rep G$ and if $\varepsilon$ is the identity then $\Rep (\mZ/2,\varepsilon)=\sVec$. 

When we write $\varepsilon: \mZ/2\to GL_n$ it stands for the homomorphism that sends the generator to diagonal matrix with entries $-1$. Let $O_m<GL_m$ and $Sp_{2n}<GL_{2n}$ denote subgroups fixing a non-degenerate symmetric and skew symmetric form. The above homomorphism $\varepsilon$ restricts to $\varepsilon: \mZ/2\to O_m$ and $\varepsilon:\mZ/2\to Sp_{2n}$.

We will use $\otimes$ also for the coproduct in the category of 2-rigs. We will only need it for $\cC\otimes\sVec$, for a 2-rig $\cC$, which is just the category of $\mZ/2$-graded objects in $\cC$, which can be interpreted as objects $X\otimes\unit\oplus Y\otimes \bar{\unit}$, where $\sigma_{Y\otimes \bar{\unit},Z\otimes \bar{\unit}}$ corresponds to $-\sigma_{Y,Z}$ (using $Y\otimes \bar{\unit}\otimes Z\otimes \bar{\unit}\simeq Y\otimes Z$). In particular, $\Rep(G,\varepsilon)$ lives naturally in $(\Rep G)\otimes\sVec$, where the latter is just the category of all representations of $G$ in $\sVec$.

%
%
%\begin{remark}
%Above we used `universal' rather loosely. $\Sym$ is universal in the sense that, for any 2-rig $\cC$ over the same field $\bk$, there is an equivalence of categories between the category of $\bk$-linear symmetric monoidal functors $F:\Sym\to \cC$ and $\cC$, given by $F\mapsto F(\bx)$. On the other hand,  Similarly, in Theorem~\ref{ThmMain} below, part (1) relates to all morphisms while (2) relates to isomorphisms.
%\end{remark}

\section{Results}

\subsection{Invertible exterior powers}
Fix $n\in\mZ_{>1}$.
Let $K$ be a commutative ring in which $n!$ is invertible. Let $\cC$ be a $K$-linear symmetric monoidal category and consider $X\in\cC$. The following result is undoubtedly known, a proof for the special case $\wedge^n X\simeq \unit$ is given in \cite[Lemma~3.6]{DR}.

\begin{theorem}\label{ThmRig}
If $\wedge^nX$ is invertible, then $X$ is rigid.
\end{theorem}

\begin{example}
Let $R$ be a commutative ring with $R$-module $M$. If $\wedge^n M$, defined as a quotient of $M^{\otimes^n_R}$, is finitely generated projective of rank~1, then $M$ is finitely generated projective and of constant rank, by \cite[Theorem~2.4]{Ram} or \cite[Proposition~3.1.1]{Sc}. Theorem~\ref{ThmRig} gives an alternative proof of this fact, assuming that $n!$ is invertible in $R$.
\end{example}

\begin{proof}[Proof of Theorem~\ref{ThmRig}]

We denote by $L$ the dual of $\wedge^n X$ and set $Y:=L\otimes \wedge^{n-1}X$.
We define $\epsilon:Y\otimes X\to\unit$ as the composite
$$L\otimes \wedge^{n-1}X \otimes X \xrightarrow{L\otimes\pi}L\otimes \wedge^n X\xrightarrow{\ev_{\wedge^nX}}\unit,  $$
with $\pi=e_n\circ \id\circ (e_{n-1}\otimes X)$. We also define $\delta:\unit\to X\otimes Y$ as the composite
$$\unit\xrightarrow{\co_{\wedge^n X}}\wedge^n X\otimes L\xrightarrow{\iota\otimes L}\wedge^{n-1}X\otimes X\otimes L\xrightarrow{\sigma_{\wedge^{n-1}X,X\otimes L}}X\otimes L\otimes \wedge^{n-1}X,$$
with $\iota=(e_{n-1}\otimes X)\circ \id\circ e_n$. In particular, we have
\begin{equation}\label{ipi}
\iota\circ\pi=
\end{equation}
$$ \frac{1}{n}\left(e_{n-1}\otimes X+(1-n)(e_{n-1}\otimes X)\circ (X^{\otimes (n-2)}\otimes \sigma_{X,X})\circ (e_{n-1}\otimes X)\right).$$

We will show that, up to a potential renormalisation, $\epsilon$ and $\delta$ satisfy the snake relations.
We thus consider the composites
$$\phi: X\xrightarrow{\delta\otimes X} X\otimes Y\otimes X\xrightarrow{X\otimes\epsilon}X\quad\mbox{and}\quad\psi:Y\xrightarrow{Y\otimes \delta}Y\otimes X\otimes Y\xrightarrow{\epsilon\otimes Y}Y.$$

 We start by computing $\phi^2$, a computation which is most intuitively considered diagrammatically, see Appendix~\ref{app}. Using the fact that $L$ is invertible to replace $\id_{L\otimes L}$ by $\pm \sigma_{L,L}$, we find
$$\phi^2\;=\;\pm A\circ B\circ C$$
with $B=L\otimes e_n\otimes X$ and
$$A=(\ev_{\wedge^n X}\otimes X)\circ (L\otimes e_n\otimes X)\circ (L\otimes e_{n-1}\otimes \sigma_{X,X})$$
and
$$C=(L\otimes e_{n-1}\otimes \sigma_{X,X})\circ(L\otimes e_n\otimes X)\circ( \co_L\otimes X). $$
Subsequently applying \eqref{ipi} then shows
\begin{equation}\label{phi2}
\phi^2\;=\; \frac{1}{n}\id_X \pm \frac{1-n}{n}\phi.
\end{equation}

Hence $\phi$ admits two-sided inverse $\phi^{-1}=n\phi\pm (n-1)\id_X$. We now set 
$$\ev_X:=\epsilon\circ(Y\otimes \phi^{-1})\quad\mbox{and}\quad \co_X:=\delta.$$
Then by construction
$$(X\otimes \ev_X)\circ (\co_X\otimes X)\;=\;\id_X.$$
On the other hand, by plugging in the explicit form of $\phi^{-1}$, we find
$$(\ev_X\otimes Y)\circ (Y\otimes \co_X)\;=\;n\psi^2\pm(n-1)\psi\;=\;\id_Y,$$
where the last equality follows from the immediate analogue of \eqref{phi2} for~$\psi$. It follows that $Y$ is the dual of $X$.
\end{proof}
%%%%%%%%%%%%%%%%%%%%%%%%%%%%%%%%%%%%%%%%%%%%%%%%%%%%%%%%%%%%%%%%%%%%%%%%%%%%%%%%%%%%%%%%%%%%%%%%%%%%%%%%%%%%%%%%%%%%%%%%%%%%%%%%%%%%%%%%%%%%%%%%%%%%%%%%%%%%%%%%%%%%%%%%%%%%%%%%%%%%%%%%%%%%%%%%%%%%%%%%%%%%%%%%%%%%%%%%%%%%%%%%%%%%%%%%%%%%%%%%%%%%%%%%%%%%%%%%%%%%%%%%%%%%%%%%%%%%%%%%%%%%%%%%%%%%%%%%%%%%%%%%%%%%%%%%%%%%%%%%%%%%%%%%%%%%%%%%%%%%%%%%%%%%%%%%%%%%%%%%%%%%%%%%%%%%%%%%%%%%%%%%%%%%%%%%%%%%%%%%%%%%%%%%%%%%%%%%%%%%%%%%%%%%%%%%%%%%%%%%%%%%%%%%%%%%%%%%%%%%%%%%%%%%%%%%%%%%%%%%%%%%%%%%%%%%%%%%%%%%%%%%%%%%%%%%%%%%%%%%%%%%%%%%%%%%%%%%%%%%%%%%%%%%%%%%%%%%%%%%%%%%%%%%%%%%%%%%%%%%%%%%%%%%%%%%%%%%%%%%

\subsection{Bosonic dimensions}
Now let $\bk$ be a field of characteristic zero and let $\cC$ be a 2-rig over $\bk$.
The following definition was suggested in \cite[\S 8]{BMT}.

\begin{definition}\label{DefDim}
\begin{enumerate}
\item An object $X\in\cC$ has bosonic subdimension $n\in\mZ_{\ge 0}$ if $\wedge^{n+1}X=0$.
\item An object $X\in\cC$ has bosonic dimension $n\in\mZ_{>0}$ if it is not a fermionic line object and $\wedge^{n}X$ is a bosonic line object.
\end{enumerate}
\end{definition}

\begin{example}\label{RemMn}
Lemma~\ref{LemMn} states that $\Rep_{\bk}M_n$ is the universal 2-rig on an object of bosonic subdimension $n$.
\end{example}

\begin{remark}\label{RemDef}\begin{enumerate}
\item Either an object has no bosonic subdimensions, or its bosonic subdimensions form a half-infinite interval. As follows from Theorem~\ref{ThmMain}(3) below, a bosonic dimension is unique (in fact, it must equal the categorical dimension by Theorem~\ref{ThmMain}(2)), see also part (2) of this remark. 
\item In \cite[Definition~8.2]{BMT}, the condition that $X$ not be a fermionic line element was not stated. Without it, every fermionic line object would have bosonic dimension $2m$ for every $m\in \mZ_{>0}$, and \cite[Conjecture~8.6]{BMT} would require an exception.
\item Having excluded fermionic line objects in Definition~\ref{DefDim}(2), one can now actually drop the demand that $\wedge^n X$ be a {\em bosonic} line object, as is clear from the proof of Theorem~\ref{ThmMain}(2).
\end{enumerate}
\end{remark}

The following confirms \cite[Conjectures~8.6 and 8.8]{BMT}. 
\begin{theorem}\label{ThmMain}
\begin{enumerate}
\item $\Rep_{\bk} GL_n$ is the universal 2-rig\footnote{As in Footnote 1, this means that there is an equivalence from the category (groupoid) of $\bk$-linear symmetric monoidal functors $\Rep_{\bk} GL_n\to\cC$, and natural transformations (which are automatically natural isomorphisms since $\Rep_{\bk} GL_n$ is rigid) between them, to the category of objects in $\cC$ of bosonic dimension $n$, and isomorphisms between them. The equivalence is given by evaluation on the vector representation of $GL_n$.} on an object of bosonic dimension~$n$.
\item If an object has bosonic dimension $n$, then it also has bosonic subdimension $n$.
\end{enumerate}
\end{theorem}
\begin{proof}

For (1), let $X$ be an object in a 2-rig $\cC$ with $\wedge^n X$ invertible (with Remark~\ref{RemDef}(3) in mind, we do not impose it is bosonic). By Theorem~\ref{ThmRig}, $X$ is rigid. Set $t:=\dim X\in\End(\unit)$. By $\dim(\wedge^n X)=\pm 1$ and \cite[(7.1.2)]{Del90},
\begin{equation}\label{eqt}\frac{t(t-1)\cdots (t-n+1)}{n!}\;=\; \pm 1.\end{equation}
So $t$ generates a finite field extesnsion $\bk\subset\bk'\subset\End(\unit)$. Since $\cC$ is automatically $\End(\unit)$-linear, there is a unique $\bk'$-linear symmetric monoidal functor $F:(\Rep_{\bk'} GL)_t\to \cC$ with $\bx_t\mapsto X$. Now $F$ is not faithful, since $\wedge^2(\wedge^n X)=0$ or $S^2(\wedge^n X)=0$, forcing $t\in\mZ$, see for instance \cite[6.2.3]{Selecta}, and in particular $\bk'=\bk$. Moreover, by \eqref{eqt}, either $t=n$ or $t=-1$. We consider first the case $t=n$. We already observed that non-zero morphisms in $\End(\bx_n^{\otimes 2n})$ are sent to zero in~$\cC$. Since $2n<2(n+1)$, the kernel of $F$ must be the unique maximal tensor ideal $\cI$, by \cite[Theorems~7.2.1(ii) and ~8.2.1(i)]{Selecta}, for which $(\Rep GL)_n/\cI\simeq \Rep GL_n$, see \cite[Th\'eor\`eme~10.4]{Del07}. 

To conclude the proof of part (2), we thus only need to exclude the case $t=-1$. Since both $\wedge^2(\wedge^n X)$ and $S^2(\wedge^n X)$ are direct sums of $\mS_\lambda X$ (see \cite[\S 1.4]{Del02} for the definition of the Schur functors $\mS_\lambda$) for partitions $\lambda\vdash 2n$ with $\lambda_1\le 2$, it follows again, now by \cite[Theorem~7.2.1(i)]{Selecta}, that the kernel of $F$ is the maximal ideal $\cI$. However, in this case \cite[Th\'eor\`eme~10.4]{Del07} states that $(\Rep GL)_{-1}/\cI$ is the category of $\mZ$-graded super vector spaces with sign twisted symmetric braiding (the representation category of the supergroup $GL_{0|1}$), so that $\bx_{-1}$ is a fermionic line object in the quotient. This is the case exludeded in Definition~\ref{DefDim}(2).

Part (2) follows immediately from the combination of part (1), Example~\ref{RemMn} and the inclusion $\Rep M_n\subset\Rep GL_n$.
\end{proof}

\begin{remark}
The analogue of Theorem~\ref{ThmMain}(3) is known for $\cC$ abelian and rigid, in which case $\bk$ is allowed to be of arbitrary characteristic, see~\cite[Proposition~2.3.2]{CEN}.
\end{remark}

\begin{corollary}
The following conditions are equivalent on $X\in\cC$ and $n\in\mZ_{>0}$.
\begin{enumerate}
\item $X$ has bosonic dimension $n$;
\item $X$ has bosonic subdimension $n$ and $\wedge^n X$ is invertible;
\item $X$ has categorical dimension $n$ and $\wedge^n X$ is invertible;
\item $X$ has bosonic subdimension $n$ and is rigid with categorical dimension $n$.
\end{enumerate}
\end{corollary}
\begin{proof}
That (1) implies the other conditions follows from Theorem~\ref{ThmMain}. That (2) implies (1), resp. (3) implies (1), follows by definition and the fact that fermionic line objects do not have finite bosonic subdimension, resp. have categorical dimension $-1\not=n$.

Finally, that (4) implies (1) follows from the proof of Theorem~\ref{ThmMain}. Indeed, if $X$ has categorical dimension $n$, we have the corresponding symmetric monoidal functor $F:(\Rep GL)_n\to\cC$. As argued in the proof of Theorem~\ref{ThmMain}, the condition $\wedge^{n+1}X=0$ then implies that $F$ factors through the quotient $\Rep GL_n$.
\end{proof}

\subsection{The fermionic case}

\begin{definition}\label{DefDim2}
\begin{enumerate}
\item An object $X\in\cC$ has fermionic subdimension $n\in\mZ_{\ge 0}$ if $S^{n+1}X=0$.
\item A non-zero object $X\in\cC$ has fermionic dimension $n\in\mZ_{>0}$ if it is not a bosonic line object and $S^{n}X$ is invertible.
\end{enumerate}
\end{definition}

With Definition~\ref{DefDim2}(2), which deviates again slightly from \cite[Definition~8.2]{BMT}, it now follows, for $X$ of fermionic dimension $n$, that $S^{n}X$ is bosonic if $n$ is even and fermionic if $n$ is odd. The proof of the following theorem is identical to that of Theorem~\ref{ThmMain}. It implies the second case in \cite[Conjecture~8.6]{BMT}.
\begin{theorem}\label{ThmMain2}
\begin{enumerate}
\item $\Rep_{\bk} (GL_n,\varepsilon)$ is the universal 2-rig on an object of fermionic dimension~$n$, its subcategory of polynomial representations is the universal 2-rig on an object of fermionic subdimension~$n$.
\item If an object has fermionic dimension $n$, then it also has fermionic subdimension $n$, and categorical dimension $-n$.
\end{enumerate}

\end{theorem}

\subsection{Variations}
We conclude this note with some refinements of the universal properties as suggested in \cite{BT}.

Parts (1) and (4) of the following theorem confirm \cite[Conjectures~35 and~36]{BT}.
\begin{theorem}\label{ThmOSp}
\begin{enumerate}
\item $\Rep_{\bk}O_m$ is the universal 2-rig on a symmetrically self-dual object of bosonic dimension $m\in\mZ_{>0}$.
\item $\Rep_{\bk}(Sp_{2n},\varepsilon)$ is the universal 2-rig on a symmetrically self-dual object of fermionic dimension $2n$ ($n\in\mZ_{>0}$).
\item $\Rep_{\bk}(O_m,\varepsilon)$ is the universal 2-rig on a skew symmetricaly self-dual object of fermionic dimension $m\in\mZ_{>0}$.
\item $\Rep_{\bk}Sp_{2n}$ is the universal 2-rig on a skew symmetrically self-dual object of bosonic dimension $2n$ ($n\in\mZ_{>0}$).
\end{enumerate}
\end{theorem}
\begin{proof}
We start with part (1). Let $X\in\cC$ be a symmetrically self-dual object of bosonic dimension $m$. Then it is also of categorical dimension $m$ and hence there is a unique symmetric monoidal functor $(\Rep O)_m\to\cC$ with $\bv_m\mapsto X$. It suffices to observe that the kernel of this functor is the unique maximal tensor ideal, so that the conclusion follows from \cite[Th\'eor\`eme~9.6]{Del07}. For this observation, we can consider $(\Rep GL)_m\to(\Rep O)_m$, $\bx_m\mapsto \bv_m$. By the proof of Theorem~\ref{ThmMain}, the kernel of the composite $(\Rep GL)_m\to \cC$ is the maximal tensor ideal in $(\Rep GL)_m$. It follows from \cite[Theorems~7.1.1(ii) and~7.2.1(ii)]{Selecta} (and faithfulness of restriction functors between supergroup representation categories) that the only tensor ideal in $(\Rep O)_m$ which has as preimage the maximal tensor ideal in $(\Rep GL)_m$ is precisely the maximal tensor ideal in  $(\Rep O)_m$, which concludes the argument.

Part (2) follows from an almost identical argument, since a symmetrically self-dual object $X\in\cC$ be of fermionic dimension $2n$ leads, via Theorem~\ref{ThmMain2}, to a unique symmetric monoidal functor $(\Rep O)_{-2n}\to\cC$ with $\bv_{-2n}\mapsto X$, so we can apply the exact same theorems in \cite{Del07, Selecta}.

Part (3) is a standard consequence of part (1). Indeed, let $X\in\cC$ be a skew symmetrically self-dual object of fermionic dimension $m$. Then $X\otimes \bar{\unit}$ in $\cC\otimes\sVec$ is a symmetrically self-dual object of bosonic dimension $m$, so that by part (1) we have a symmetric monoidal functor
$$\Rep O_m\;\to\;\cC\otimes\sVec,\quad \bv_m\mapsto X\otimes\bar{\unit},$$
where, with slight abuse of notation, we write $\bv_m$ for the vector representation of $O_m$. We can combine this with the inclusion of $\sVec$, to obtain a
symmetric monoidal functor
$$\Rep O_m\otimes \sVec\;\to\;\cC\otimes\sVec,\quad \bv_m\otimes\bar{\unit}\mapsto X\otimes{\unit},$$
no longer defined by the the prescription on the right-hand side. Now $\Rep (O_m,\varepsilon)$ is precisely the symmetric monoidal subcategory of $\Rep O_m\otimes \sVec$ generated by $\bv_m':=\bv_m\otimes\bar{\unit}$. Hence the functor that sends a symmetric monoidal functor $F:\Rep (O_m,\varepsilon)\to\cC$ to $F(\bv_m')$ is an essentially surjective functor from the category of such monoidal functors to the category of skew symmetrically self-dual object of fermionic dimension $m$ (and isomorphisms) in $\cC$. The same reasoning shows it is fully faithful.

Part (4) follows from part (2) identically to how (3) follows from~(1).
\end{proof}

\begin{remark}
It follows from the proof (mainly \cite[Th\'eor\`eme~9.6]{Del07}) that the `remaining cases' in Theorem~\ref{ThmOSp} do not occur. For example, a symmetrically self-dual object in a 2-rig cannot have an odd fermionic dimension. One can see this directly. If $X$ is symmetrically self-dual, then so $S^n X$. So if $X$ has fermionic dimension $2m+1$ then the fermionic line object $S^{2m+1}X$ is supposed to be symmetrically self-dual, a contradiction.
\end{remark}

The following interpretation of \cite[Conjecture~37]{BT} is a consequence of the diagrammatic presentation from \cite{LZ}. We denote by $\bv_m$ the natural representation of $SO_m$, and by $\Delta:\unit\to \wedge^m\bv_m$ the morphism that sends $1$ to the skew symmetrisation of $e_1\otimes e_2\otimes\cdots\otimes e_m$, for a basis $\{e_1,\ldots, e_m\}$ of $\bv_m$. 
For a morphism $\alpha:X\to Y$ in a 2-rig we write $\alpha^\ast:Y^\ast\to X^\ast$ for the adjoint morphism.

\begin{theorem}\label{ThmSO}
For a 2-rig $\cC$ over $\bk$, the assignment that sends a symmetric monoidal functor $F:\Rep SO_m\to\cC$ to the pair of object $F(\bv_m)$ and morphism $F(\Delta):\unit\to \wedge^m F(\bv_m)$ yields an equivalence between the category of such symmetric monoidal functors and the category of pairs $(X,\alpha)$ of $X\in\cC$ and $\alpha:\unit\to \wedge^mX$, where $X$ is a symmetrically self-dual object and $\alpha$ is an isomorphism with inverse given by its normalised conjugate $\alpha^\ast/m!$. A morphism $(X,\alpha)\to(Y,\beta)$ is an isomorphism $X\to Y$ producing a commutative triangle with $\alpha$ and $\beta$.
\end{theorem}
\begin{proof}
Let $\cB(m)$ the Brauer category over $\bk$ at parameter $m$ (which has $(\Rep O)_m$ as Cauchy completion). The enhanced Brauer category $\widetilde{\cB}(m)$ from \cite[\S 5]{LZ} is obtained by adjoining a single morphism $\Delta_m:\unit\to \bv_m^{\otimes m}$ to $\cB(m)$ satisfying the relations
$$\Delta_m\circ \Delta_m^\ast\;=\; a_m\quad\mbox{and}\quad \Delta^\ast_m\circ \Delta_m=m!\cdot\id_{\unit}.$$
More precisely, the first relation is (4) in \cite[Definition~5.1]{LZ}, while the second relation is derived in \cite[Lemma~5.4]{LZ} from the defining relations. Conversely, the missing relations (2), (3) in \cite[Definition~5.1]{LZ} can be derived from the above relations.

Building further on the universal property of $\cB(m)$ in \cite[Proposition~9.4]{Del07}, it thus follows that the universal property of $\widetilde{\cB}(m)$ is such that it classifies morphisms $\beta:\unit\to Y^{\otimes m}$ where $Y$ is symmetrically self-dual (of categorical dimension $m$) and $\beta\circ\beta^\ast =a_m$ and $\beta^\ast\circ \beta=m!$. In particular $\beta=e_m\circ \beta$ and if we are only dealing with idempotent complete categories we can reinterpret the universal property of the Cauchy completion of $\widetilde{\cB}(m)$ as classifying morphisms $\beta:\unit\to\wedge^m Y$ with $\beta^\ast\circ\beta =m!$ and $\beta\circ \beta^\ast=m!$.

We have thus proved that the Cauchy completion of $\widetilde{\cB}(m)$ has the universal property desired from $\Rep SO_m$, and by \cite[Theorem~6.1]{LZ}, the two are equivalent.
\end{proof}

The universal property for $SL_n$ can be proved similarly to the one for $SO_m$ above, using diagrammatic presentations. However, we choose a categorical proof, to display some more methods. Of course, we could also prove the universal property for $SO_m$ as in the proof below. The method then actually simplifies, since $O_m/SO_m$ is finite, contrary to $GL_n/SL_n$.

The following proves \cite[Conjecture~34]{BT}.  We denote by $\bx_n$ the natural representation of $SL_n$ (and of $GL_n$), and by $\Delta:\unit\to \wedge^n\bx_n$ the morphism that sends $1$ to the skew symmetrisation of $e_1\otimes e_2\otimes\cdots\otimes e_n$. 
\begin{theorem}
For a 2-rig $\cC$ over $\bk$, the assignment that sends a symmetric monoidal functor $F:\Rep SL_n\to\cC$ to the pair of object $F(\bx_n)$ and morphism $F(\Delta):\unit\to \wedge^nF(\bx_n)$ yields an equivalence between the category of such symmetric monoidal functors and the category of pairs $(X,\alpha)$ of $X\in\cC$ and isomorphism $\alpha:\unit\to \wedge^nX$. A morphism $(X,\alpha)\to(Y,\beta)$ is an isomorphism $X\to Y$ producing a commutative triangle with $\alpha$ and $\beta$.
\end{theorem}
\begin{proof}
As $SL_n<GL_n$ is a normal closed subgroup, with quotient the multiplicative group $\mG_m$, it follows that the category of all (not just finite dimensional) rational representations $\Rep^\infty SL_n$ equivalent to the category of $A$-modules in $\Rep^\infty GL_n$, where
$$A\;\simeq\;\bk[x,x^{-1}]\;\simeq\;\cO(\mG_m)\;\simeq\;\cO(GL_n)^{SL_n}.$$
This is an algebra in $\Rep^\infty GL_n$, for the left regular action of $GL_n$ on the (right) invariants $\cO(GL_n)^{SL_n}$. This means that $x$ is to be identified with the invertible object $\wedge^n\bv_n$. A proof of this general principle follows for instance from \cite[Lemma~6.2.1]{Incomp} in combination with the main result of \cite{CPS}.

Since the representation categories of $SL_n$ and $GL_n$ are semisimple, the infinite representation categories are both the ind-completions as well as the presheaf categories, e.g.
$$\Rep^\infty SL_n\;\simeq\;\Ind\Rep SL_n\;\simeq\; \PSh\Rep SL_n,$$
where for a $\bk$-linear category $\cA$ we denote by $\PSh\cA$ the presheaf category of $\bk$-linear functors $\cA^{\op}\to\Vecc^\infty$. If $\cA$ is (symmetric) monoidal, then so is $\PSh\cA$ via Day convolution.

Let us now first assume that $\cC$ is cocomplete. Then by \cite[Theorem~5.1]{IK} the category of symmetric monoidal functors $\Rep SL_n\to\cC$ is equivalent with the category of cocontinuous symmetric monoidal functors $\Rep^\infty SL_n\to\cC$. The latter is equivalent, by \cite[Proposition~5.3.1]{Br}, with the category of pairs of a cocontinuous symmetric monoidal functor $F:\Rep^\infty GL_n\to\cC$ and an algebra morphism $F(A)\to\unit$. By our simple form of the algebra $A$, the data of such an algebra morphism is equivalent to the choice of an isomorphism $F(\wedge^n\bx_n)=F(x)\to\unit$. In conclusion, we find an equivalence between the category of symmetric monoidal functors $\Rep SL_n\to\cC$ and the category of pairs of a symmetric monoidal functor $H:\Rep GL_n\to\cC$ and an isomorphism $H(\wedge^n \bx_n)\simeq \unit$. We can thus invoke Theorem~\ref{ThmMain}(1) and one can trace through the argument to verify the equivalence corresponds to the one spelled out in the theorem.

Now, in case $\cC$ is not cocomplete, we can replace it with the cocomplete $\PSh\cC$ and observe that the category of symmetric monoidal functors $\Rep SL_n\to\cC$ is equivalent with the full subcategory of symmetric monoidal functors $\Rep SL_n\to\PSh\cC$ that happen to take values in (equivalently, send $\bx_n$ into) $\cC$ (under the Yoneda embedding $\cC\subset\PSh\cC$).
\end{proof}

\begin{remark}We point out the asymmetry between the universal properties for $SL_n$ and $SO_m$, {\it i.e.} the need to specify the inverse of the isomorphism in Theorem~\ref{ThmSO}, which relates to the choices involved in picking a dual.
\end{remark}

%%%%%%%%%%%%%%%%%%%%%%%%%%%%%%%%%%%%%%%%%%%%%%%%%%%%%%%%%%%%%%%%%%%%%%%%%%%%%%%%%%%%%%%%%%%%%%%%%%%%%%%%%%%%%%%%%%%%%%%%%%%%%%%%%%%%%%%%%%%%%%%%%%%%%%%%%%%%%%%%%%%%%%%%%%%%%%%%%%%%%%%%%%%%%%%%%%%%%%%%%%%%%%%%%%%%%%%%%%%%%%%%%%%%%%%%%%%%%%%%%%%%%%%%%%%%%%%%%%%%%%%%%%%%%%%%%%%%%%%%%%%%%%%%%%%%%%%%%%%%%%%%%%%%%%%%%%%%%%%%%%%%%%%%%%%%%%%%%%%%%%%%%%%%%%%%%%%%%%%%%%%%%%%%%%%%%%%%%%%%%%%%%%%%%%%%%%%%%%%%%%%%%%%%%%%%%%%%%%%%%%%%%%%%%%%%%%%%%%%%%%%%%%%%%%%%%%%%%%%%%%%%%%%%%%%%%%%%%%%%%%%%%%%%%%%%%%%%%%%%%%%%%%%%%%%%%%%%%%%%%%%%%%%%%%%%%%%%%%%%%%%%%%%%%%%%%%%%%%%%%%%%%%%%%%%%%%%%%%%%%%%%%%%%%%%%%%%%%%%%

\appendix

\section{Diagrammatic proof of equation~\eqref{phi2}}\label{app}

We consider the case $n=3$ and assume that $\wedge^3X$ is a bosonic line object. We draw identity morphisms of $X$ with solid lines and identity morphisms of $\wedge^3X$ and its dual $L$ by dashed lines. Diagrams go from top to bottom, crossings represent braiding isomorphisms and (dashed) cups and caps represent (co)evaluations (of invertible objects). We denote by rectangular boxes the idempotent $e_3$. We remove the idempotents $e_2\otimes X$ from the diagrams, which play no role as they are all sandwiched between idempotents of the form $e_3$.
 We can then represent $\phi^2$, and its rewritten version using $\sigma_{L,L}=\id$, as

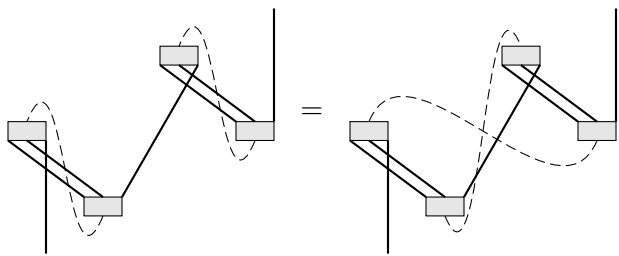
\vspace{-7mm}

\begin{center}

\begin{tikzpicture}[scale=0.5]
% Rectangle dimensions
\def\rectwidth{1}
\def\rectheight{0.5}

% Define center positions of rectangles
\coordinate (A) at (4,0);
\coordinate (B) at (2,2);
\coordinate (C) at (8,2);
\coordinate (D) at (6,4);

% Helper macro to draw a horizontal rectangle centered at a point
\newcommand{\rect}[1]{
  \draw[fill=gray!20]
    (#1) ++(-\rectwidth/2,-\rectheight/2)
    rectangle ++(\rectwidth,\rectheight);
}

% Draw the rectangles
\rect{A}
\rect{B}
\rect{C}
\rect{D}

% Compute rectangle edge points for connections
\path (A) ++(-\rectwidth/2,\rectheight/2) coordinate (A-1);
\path (A) ++(0,\rectheight/2) coordinate (A-2);
\path (A) ++(\rectwidth/2,\rectheight/2) coordinate (A-4);

\path (B) ++(-\rectwidth/2,-\rectheight/2) coordinate (B-1);
\path (B) ++(0,-\rectheight/2) coordinate (B-2);
\path (B) ++(\rectwidth/2,-\rectheight/2) coordinate (B-4);
\path (B) ++(\rectwidth/2,-\rectheight/2-3) coordinate (BB);

\path (C) ++(-\rectwidth/2,\rectheight/2) coordinate (C-1);
\path (C) ++(0,\rectheight/2) coordinate (C-2);
\path (C) ++(\rectwidth/2,\rectheight/2) coordinate (C-4);
\path (C) ++(\rectwidth/2,\rectheight/2+3) coordinate (CC);

\path (D) ++(-\rectwidth/2,-\rectheight/2) coordinate (D-1);
\path (D) ++(0,-\rectheight/2) coordinate (D-2);
\path (D) ++(\rectwidth/2,-\rectheight/2) coordinate (D-4);

\path (C) ++(\rectwidth/2,0) coordinate (C-right);
\path (C) ++(-\rectwidth/2,0) coordinate (C-left);

\path (D) ++(\rectwidth/2,0) coordinate (D-right);
\path (D) ++(-\rectwidth/2,0) coordinate (D-left);

% Curved connections (enter/exit from bottom/top)
\path (A) ++(0,-\rectheight/2) coordinate (A-3);
\path (C) ++(0,-\rectheight/2) coordinate (C-3);
\path (C) ++(0,\rectheight/2) coordinate (C-top);

\path (B) ++(0,+\rectheight/2) coordinate (B-3);
\path (D) ++(0,+\rectheight/2) coordinate (D-3);

% Draw connections
\draw[thick] (A-1) -- (B-1);
\draw[thick] (A-2) -- (B-2);
\draw[thick] (C-1) -- (D-1);
\draw[thick] (C-2) -- (D-2);
\draw[thick] (A-4) -- (D-4);
\draw[thick] (B-4) -- (BB);
\draw[thick] (C-4) -- (CC);

\draw[dash pattern=on 4pt off 2pt]
  (A-3) .. controls +(-1,-2.5) and +(1,2.5) .. (B-3);
  \draw[dash pattern=on 4pt off 2pt]
  (C-3) .. controls +(-1,-2.5) and +(1,2.5) .. (D-3);

\node at (9.5,2.5) {$=$};

% Define center positions of rectangles
\coordinate (E) at (13,0);
\coordinate (F) at (11,2);
\coordinate (G) at (17,2);
\coordinate (H) at (15,4);

% Draw the rectangles
\rect{E}
\rect{F}
\rect{G}
\rect{H}

% Compute rectangle edge points for connections
\path (E) ++(-\rectwidth/2,\rectheight/2) coordinate (E-1);
\path (E) ++(0,\rectheight/2) coordinate (E-2);
\path (E) ++(\rectwidth/2,\rectheight/2) coordinate (E-4);

\path (F) ++(-\rectwidth/2,-\rectheight/2) coordinate (F-1);
\path (F) ++(0,-\rectheight/2) coordinate (F-2);
\path (F) ++(\rectwidth/2,-\rectheight/2) coordinate (F-4);
\path (F) ++(\rectwidth/2,-\rectheight/2-3) coordinate (FF);

\path (G) ++(-\rectwidth/2,\rectheight/2) coordinate (G-1);
\path (G) ++(0,\rectheight/2) coordinate (G-2);
\path (G) ++(\rectwidth/2,\rectheight/2) coordinate (G-4);
\path (G) ++(\rectwidth/2,\rectheight/2+3) coordinate (GG);

\path (H) ++(-\rectwidth/2,-\rectheight/2) coordinate (H-1);
\path (H) ++(0,-\rectheight/2) coordinate (H-2);
\path (H) ++(\rectwidth/2,-\rectheight/2) coordinate (H-4);

\path (G) ++(\rectwidth/2,0) coordinate (G-right);
\path (G) ++(-\rectwidth/2,0) coordinate (G-left);

\path (H) ++(\rectwidth/2,0) coordinate (H-right);
\path (H) ++(-\rectwidth/2,0) coordinate (H-left);

% Curved connections (enter/exit from bottom/top)
\path (E) ++(0,-\rectheight/2) coordinate (E-3);
\path (G) ++(0,-\rectheight/2) coordinate (G-3);
\path (G) ++(0,\rectheight/2) coordinate (G-top);

\path (F) ++(0,+\rectheight/2) coordinate (F-3);
\path (H) ++(0,+\rectheight/2) coordinate (H-3);

% Draw connections
\draw[thick] (E-1) -- (F-1);
\draw[thick] (E-2) -- (F-2);
\draw[thick] (G-1) -- (H-1);
\draw[thick] (G-2) -- (H-2);
\draw[thick] (E-4) -- (H-4);
\draw[thick] (F-4) -- (FF);
\draw[thick] (G-4) -- (GG);

\draw[dash pattern=on 4pt off 2pt]
  (E-3) .. controls +(1,-2.5) and +(-1,2.5) .. (H-3);
\draw[dash pattern=on 4pt off 2pt]
  (G-3) .. controls +(-1,-2.5) and +(1,2.5) .. (F-3);
\end{tikzpicture}
\end{center}

\vspace{-7mm}

\noindent
Applying the snake relation for $L$ then allows us to simplify the above into

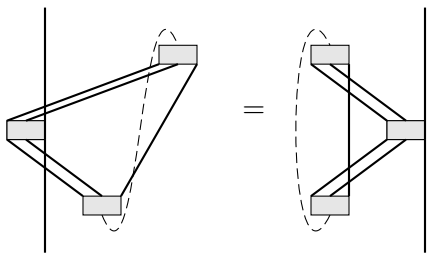
\vspace{-7mm}

\begin{center}

\begin{tikzpicture}[scale=0.5]

% Rectangle dimensions
\def\rectwidth{1}
\def\rectheight{0.5}

% Define center positions of rectangles
\coordinate (A) at (4,0);
\coordinate (B) at (2,2);
\coordinate (C) at (2,2);
\coordinate (D) at (6,4);

% Helper macro to draw a horizontal rectangle centered at a point
\newcommand{\rect}[1]{
  \draw[fill=gray!20]
    (#1) ++(-\rectwidth/2,-\rectheight/2)
    rectangle ++(\rectwidth,\rectheight);
}

% Draw the rectangles
\rect{A}
\rect{B}
\rect{C}
\rect{D}

% Compute rectangle edge points for connections
\path (A) ++(-\rectwidth/2,\rectheight/2) coordinate (A-1);
\path (A) ++(0,\rectheight/2) coordinate (A-2);
\path (A) ++(\rectwidth/2,\rectheight/2) coordinate (A-4);

\path (B) ++(-\rectwidth/2,-\rectheight/2) coordinate (B-1);
\path (B) ++(0,-\rectheight/2) coordinate (B-2);
\path (B) ++(\rectwidth/2,-\rectheight/2) coordinate (B-4);
\path (B) ++(\rectwidth/2,-\rectheight/2-3) coordinate (BB);

\path (B) ++(-\rectwidth/2,\rectheight/2) coordinate (C-1);
\path (B) ++(0,\rectheight/2) coordinate (C-2);
\path (B) ++(\rectwidth/2,\rectheight/2) coordinate (C-4);
\path (B) ++(\rectwidth/2,\rectheight/2+3) coordinate (CC);

\path (D) ++(-\rectwidth/2,-\rectheight/2) coordinate (D-1);
\path (D) ++(0,-\rectheight/2) coordinate (D-2);
\path (D) ++(\rectwidth/2,-\rectheight/2) coordinate (D-4);

\path (D) ++(\rectwidth/2,0) coordinate (D-right);
\path (D) ++(-\rectwidth/2,0) coordinate (D-left);

% Curved connections (enter/exit from bottom/top)
\path (A) ++(0,-\rectheight/2) coordinate (A-3);
\path (B) ++(0,-\rectheight/2) coordinate (C-3);

\path (B) ++(0,+\rectheight/2) coordinate (B-3);
\path (D) ++(0,+\rectheight/2) coordinate (D-3);

% Draw connections
\draw[thick] (A-1) -- (B-1);
\draw[thick] (A-2) -- (B-2);
\draw[thick] (C-1) -- (D-1);
\draw[thick] (C-2) -- (D-2);
\draw[thick] (A-4) -- (D-4);
\draw[thick] (B-4) -- (BB);
\draw[thick] (B-4) -- (CC);

\draw[dash pattern=on 4pt off 2pt]
  (A-3) .. controls +(1,-2.5) and +(-1,2.5) .. (D-3);

\node at (8,2.5) {$=$};

\coordinate (E) at (10,0);
\coordinate (F) at (12,2);
\coordinate (G) at (12,2);
\coordinate (H) at (10,4);

% Draw the rectangles
\rect{E}
\rect{F}
\rect{G}
\rect{H}

% Compute rectangle edge points for connections
\path (E) ++(-\rectwidth/2,\rectheight/2) coordinate (E-1);
\path (E) ++(0,\rectheight/2) coordinate (E-2);
\path (E) ++(\rectwidth/2,\rectheight/2) coordinate (E-4);

\path (F) ++(-\rectwidth/2,-\rectheight/2) coordinate (F-1);
\path (F) ++(0,-\rectheight/2) coordinate (F-2);
\path (F) ++(\rectwidth/2,-\rectheight/2) coordinate (F-4);
\path (F) ++(\rectwidth/2,-\rectheight/2-3) coordinate (FF);

\path (F) ++(-\rectwidth/2,\rectheight/2) coordinate (G-1);
\path (F) ++(0,\rectheight/2) coordinate (G-2);
\path (F) ++(\rectwidth/2,\rectheight/2) coordinate (G-4);
\path (F) ++(\rectwidth/2,\rectheight/2+3) coordinate (GG);

\path (H) ++(-\rectwidth/2,-\rectheight/2) coordinate (H-1);
\path (H) ++(0,-\rectheight/2) coordinate (H-2);
\path (H) ++(\rectwidth/2,-\rectheight/2) coordinate (H-4);

\path (H) ++(\rectwidth/2,0) coordinate (H-right);
\path (H) ++(-\rectwidth/2,0) coordinate (H-left);

% Curved connections (enter/exit from bottom/top)
\path (E) ++(0,-\rectheight/2) coordinate (E-3);
\path (F) ++(0,-\rectheight/2) coordinate (G-3);

\path (F) ++(0,+\rectheight/2) coordinate (F-3);
\path (H) ++(0,+\rectheight/2) coordinate (H-3);

% Draw connections
\draw[thick] (E-1) -- (F-1);
\draw[thick] (E-2) -- (F-2);
\draw[thick] (G-1) -- (H-1);
\draw[thick] (G-2) -- (H-2);
\draw[thick] (E-4) -- (H-4);
\draw[thick] (F-4) -- (FF);
\draw[thick] (F-4) -- (GG);

\draw[dash pattern=on 4pt off 2pt]
  (E-3) .. controls +(-1.2,-2.5) and +(-1.2,2.5) .. (H-3);

\end{tikzpicture}

\end{center}

\vspace{-7mm}

\noindent
Via relation \eqref{ipi} and minor manipulations this then becomes

\vspace{-7mm}

\begin{center}

\begin{tikzpicture}[scale=0.5]
% Rectangle dimensions
\def\rectwidth{1}
\def\rectheight{0.5}

% Define center positions of rectangles
\coordinate (A) at (4,0);
\coordinate (B) at (2,2);
\coordinate (C) at (2,2);
\coordinate (D) at (4,4);

% Helper macro to draw a horizontal rectangle centered at a point
\newcommand{\rect}[1]{
  \draw[fill=gray!20]
    (#1) ++(-\rectwidth/2,-\rectheight/2)
    rectangle ++(\rectwidth,\rectheight);
}

% Draw the rectangles
\rect{A}
\rect{D}

% Compute rectangle edge points for connections
\path (A) ++(-\rectwidth/2,\rectheight/2) coordinate (A-1);
\path (A) ++(0,\rectheight/2) coordinate (A-2);
\path (A) ++(\rectwidth/2,\rectheight/2) coordinate (A-4);

\path (B) ++(-\rectwidth/2,-\rectheight/2) coordinate (B-1);
\path (B) ++(0,-\rectheight/2) coordinate (B-2);
\path (B) ++(\rectwidth/2,-\rectheight/2) coordinate (B-4);
\path (B) ++(\rectwidth/2,-\rectheight/2-3) coordinate (BB);

\path (B) ++(-\rectwidth/2,\rectheight/2) coordinate (C-1);
\path (B) ++(0,\rectheight/2) coordinate (C-2);
\path (B) ++(\rectwidth/2,\rectheight/2) coordinate (C-4);
\path (B) ++(\rectwidth/2,\rectheight/2+3) coordinate (CC);

\path (D) ++(-\rectwidth/2,-\rectheight/2) coordinate (D-1);
\path (D) ++(0,-\rectheight/2) coordinate (D-2);
\path (D) ++(\rectwidth/2,-\rectheight/2) coordinate (D-4);

\path (D) ++(\rectwidth/2,0) coordinate (D-right);
\path (D) ++(-\rectwidth/2,0) coordinate (D-left);

% Curved connections (enter/exit from bottom/top)
\path (A) ++(0,-\rectheight/2) coordinate (A-3);
\path (B) ++(0,-\rectheight/2) coordinate (C-3);

\path (B) ++(0,+\rectheight/2) coordinate (B-3);
\path (D) ++(0,+\rectheight/2) coordinate (D-3);

% Draw connections
\draw[thick] (A-1) -- (D-1);
\draw[thick] (A-2) -- (D-2);
\draw[thick] (A-4) -- (D-4);
\draw[thick] (5,-1) -- (5,5);

\draw[dash pattern=on 4pt off 2pt]
  (A-3) .. controls +(-2,-2.5) and +(-2,2.5) .. (D-3);

\node at (2,2.5) {$\frac{1}{3}$};

\node at (7,2.5) {$-\quad \frac{2}{3}$};

  \coordinate (E) at (10,0);
\coordinate (F) at (8,2);
\coordinate (G) at (8,2);
\coordinate (H) at (10,4);

% Draw the rectangles
\rect{E}
\rect{H}

% Compute rectangle edge points for connections
\path (E) ++(-\rectwidth/2,\rectheight/2) coordinate (E-1);
\path (E) ++(0,\rectheight/2) coordinate (E-2);
\path (E) ++(\rectwidth/2,\rectheight/2) coordinate (E-4);

\path (F) ++(-\rectwidth/2,-\rectheight/2) coordinate (F-1);
\path (F) ++(0,-\rectheight/2) coordinate (F-2);
\path (F) ++(\rectwidth/2,-\rectheight/2) coordinate (F-4);
\path (F) ++(\rectwidth/2,-\rectheight/2-3) coordinate (FF);

\path (F) ++(-\rectwidth/2,\rectheight/2) coordinate (G-1);
\path (F) ++(0,\rectheight/2) coordinate (G-2);
\path (F) ++(\rectwidth/2,\rectheight/2) coordinate (G-4);
\path (F) ++(\rectwidth/2,\rectheight/2+3) coordinate (GG);

\path (H) ++(-\rectwidth/2,-\rectheight/2) coordinate (H-1);
\path (H) ++(0,-\rectheight/2) coordinate (H-2);
\path (H) ++(\rectwidth/2,-\rectheight/2) coordinate (H-4);

\path (H) ++(\rectwidth/2,0) coordinate (H-right);
\path (H) ++(-\rectwidth/2,0) coordinate (H-left);

% Curved connections (enter/exit from bottom/top)
\path (E) ++(0,-\rectheight/2) coordinate (E-3);
\path (F) ++(0,-\rectheight/2) coordinate (G-3);

\path (F) ++(0,+\rectheight/2) coordinate (F-3);
\path (H) ++(0,+\rectheight/2) coordinate (H-3);

% Draw connections
\draw[thick] (E-1) -- (H-1);
\draw[thick] (E-2) -- (H-2);
\draw[thick] (E-4) -- (11,5);
\draw[thick] (H-4) -- (11,-1);

\draw[dash pattern=on 4pt off 2pt]
  (E-3) .. controls +(-2,-2.5) and +(-2,2.5) .. (H-3);
  
\end{tikzpicture}\end{center}

\vspace{-7mm}

\noindent
Since $\dim(L)=1$, this is now $\frac{1}{3}\id_X-\frac{2}{3}\phi$ indeed.

\subsection*{Acknowledgement} The research was supported by ARC grant FT220100125. The author thanks John Baez for useful anddiscussions and pointing out reference \cite{BT} and Pasquale Zito for pointing out \cite{DR}.

%%%%%%%%%%%%%%%%%%%%%%%%%%%%%%%%%%%%%%%%%%%%%%%%%%%%%%%%%%%%%%%%%%%%%%%%%%%%%%%%%%%%%%%%%%%%%%%%%%%%%%%%%%%%%%%%%%%%%%%%%%%%%%%%%%%%%%%%%%%%%%%%%%%%%%%%%%%%%%%%%%%%%%%%%%%%%%%%%%%%%%%%%%%%%%%%%%%%%%%%%%%%%%%%%%%%%%%%%%%%%%%%%%%%%%%%%%%%%%%%%%%%%%%%%%%%%%%%%%%%%%%%%%%%%%%%%%%%%%%%%%%%%%%%%%%%%%%%%%%%%%%%%%%%%%%%%%%%%%%%%%%%%%%%%%%%%%%%%%%%%%%%%%%%%%%%%%%%%%%%%%%%%%%%%%%%%%%%%%%%%%%%%%%%%%%%%%%%%%%%%%%%%%%%%%%%%%%%%%%%%%%%%%%%%%%%%%%%%%%%%%%%%%%%%%%%%%%%%%%%%%%%%%%%%%%%%%%%%%%%%%%%%%%%%%%%%%%%%%%%%%%%%%%%%%%%%%%%%%%%%%%%%%%%%%%%%%%%%%%%%%%%%%%%%%%%%%%%%%%%%%%%%%%%%%%%%%%%%%%%%%%%%%%%%%%%%%%%%%%%

\end{document}